\documentclass[3p,10pt,a4paper,twoside,fleqn,sort&compress]{filomat}
\usepackage{amssymb,amsmath,latexsym}
\usepackage[bookmarksnumbered, plainpages, backref]{hyperref}
\usepackage[varg]{pxfonts}
\usepackage{algorithm}
\usepackage{algorithmic}
\usepackage{setspace}
\newtheorem{theorem}{Theorem}[section]

\newtheorem{corollary}[theorem]{Corollary}

\newtheorem{definition}[theorem]{Definition}
\newtheorem{example}[theorem]{Example}

\newtheorem{lemma}[theorem]{Lemma}

\newtheorem{proposition}[theorem]{Proposition}
\newtheorem{remark}[theorem]{Remark}

\begin{document}

\allowdisplaybreaks
\title{Monotone Relations in Hadamard Spaces}

\author[affil1]{Ali Moslemipour}
\ead{ali.moslemipour@gmail.com}
\author[affil2]{Mehdi Roohi}
\ead{m.roohi@gu.ac.ir}
\address[affil1]{Department of Mathematics, Science and Research Branch,\\ Islamic Azad University, Tehran, Iran}
\address[affil2]{Department of Mathematics, Faculty of Sciences,\\
Golestan University, Gorgan, Iran}
\newcommand{\AuthorNames}{A. Moslemipour, M. Roohi}

\newcommand{\FilMSC}{Primary 47H05; Secondary 47H04}
\newcommand{\FilKeywords}{(Monotonicity, geodesic space, flat Hadamard spaces, monotone relations, Lipschitz semi-norm.)}

\newcommand*{\loze}{{\scalebox{0.55}{$\lozenge$}}}

\vspace{0.3cm}

\begin{abstract}
In this paper,  the notion of $\mathcal{W}$-property for subsets of $X\times X^\loze$ is introduced  and investigated, where $X$ is an Hadamard space and $X^\loze$ is its linear dual space. It is shown that an Hadamard space $X$ is flat if and only if $X\times X^\loze$ has $\mathcal{W}$-property.  Moreover, the notion of monotone relation from an Hadamard space to its linear dual space is introduced. A characterization result for monotone relations with $\mathcal{W}$-property (and hence in flat Hadamard spaces) is given. Finally, a type of Debrunner-Flor Lemma  concerning extension of monotone relations in Hadamard spaces is proved.
\end{abstract}

\maketitle

\makeatletter
\renewcommand\@makefnmark%
{\mbox{\textsuperscript{\normalfont\@thefnmark)}}}
\makeatother
\newcommand{\supp}{\mathrm{supp}}
\newcommand{\Ran}{\mathrm{Range}}
\newcommand{\Dom}{\mathrm{Dom}}
\newcommand{\dom}{\mathrm{dom}}
\newcommand{\gra}{\mathrm{gra}}
\newcommand{\spa}{\mathrm{span}}
\newcommand{\Fix}{\mathrm{Fix}}
\newcommand{\co}{\mathrm{co}}
\newcommand{\cat}{\mathrm{CAT}}
\newcommand{\Id}{\mathrm{Id}}
\newcommand{\epi}{\mathrm{epi}}
\newcommand{\Int}{\mathrm{int}}
\def\N {{\mathbb N}}
\def\R {{\mathbb R}}
\newcommand*{\LargerCdot}{\raisebox{-0.95ex}{\scalebox{2.5}{$\cdot$}}}
\newcommand*{\loz}{{\scalebox{0.5}{$\lozenge$}}}
\newcommand*{\tl}{{\scalebox{1.5}{$\,\tilde{\scalebox{0.66}{\!\!{M}}}$}}}
\newcommand*{\tll}{\raisebox{-0.95ex}{\scalebox{2}{$\,\tilde{\scalebox{0.5}{\!\!\!$\mathfrak{M}$}}$}}}
\newcommand*{\shp}{{\scalebox{0.55}{$\#$}}}

\vspace{0.3cm}

\section{Introduction and Preliminaries}
Let $(X,d)$ be a metric space. We say that a mapping  $c:[0,1] \rightarrow X$ is a \textit{geodesic path} from $x \in X$ to $y \in X$ if  $c(0)=x$, $c(1)=y$ and $d(c(t),c(s))=|t-s|d(x,y),$ for each $t,s \in [0,1]$.
The image of $c$ is said to be a \textit{geodesic segment} joining $x$ and $y$.
A metric space $(X,d)$ is called a \textit{geodesic space} if there is a geodesic path  between every two points of $X$.  Also, a geodesic space $X$ is called \textit{uniquely geodesic space} if for each $x, y\in X$ there exists a unique geodesic path from $x$ to $y$. From now on, in a uniquely geodesic space, we denote the set $c([0,1])$ by $[x,y]$ and for each $z \in [x,y]$, we write $z=(1-t)x \oplus ty$, where $t\in [0,1]$. In this case, we say that $z$ is a \textit{convex combination} of $x$ and $y$. Hence, $[x,y]=\{(1-t)x \oplus ty: t \in [0,1]\}$. More details can be found in \cite{Bacak2014, BridsonHaefliger}.

\begin{definition}{\rm{\rm\cite[Definition 2.2]{DhompongsaKaewkhaoPanyanak2012}}} {\rm Let $(X,d)$ be a geodesic space, $v_1,v_2,v_3,\ldots,v_n$ be $n$ points in $X$ and $ \{\lambda_1,\lambda_2,\lambda_3,\ldots,\lambda_n\} \subseteq (0,1)$ be such that $\sum^n_{i=1}\lambda_i=1$. We  define \textit{convex combination} of $ \{v_1,v_2,v_3,\ldots,v_n\}$ inductively as following:
 \begin{align}\label{cnv}
\scalebox{1.5}{$\oplus$}^n_{i=1}\lambda_i v_i:=(1-\lambda_n)\bigg(\frac{\lambda_1}{1-\lambda_n}v_1\oplus\frac{\lambda_2}{1-\lambda_n}v_2\oplus\cdots\oplus\frac{\lambda_{n-1}}{1-\lambda_n}v_n\bigg)\oplus\lambda_nv_n.
\end{align}
Note that for every $x \in X$, we have $d \big(x,\scalebox{1.5}{$\oplus$} ^n_{i=1}\lambda_i v_i\big)\leq \sum_{i=1}^n \lambda_i d(x,v_i)$.
}\end{definition}

According to  \cite[Definition 1.2.1]{Bacak2014}, a geodesic space $(X,d) $ is a \textit{$\cat(0)$ space}, if  the following condition, so-called \textit{CN-inequality}, holds:
\begin{align} \label{cn}
d(z,(1-\lambda)x\oplus \lambda y)^2 \leq (1-\lambda)d(z,x)^2 +\lambda d(z,y)^2-\lambda(1-\lambda)d(x,y)^2 ~\text{for all}~x,y,z \in X, \lambda \in [0,1].
\end{align}
One can show that (for instance see  {\rm{\rm\cite[Theorem 1.3.3]{Bacak2014}}}) $\cat(0)$ spaces are uniquely geodesic spaces.
An \textit{Hadamard space} is a complete $\cat(0)$ space.

Let $X$ be an Hadamard space. For each $x, y\in X$, the ordered pair $(x, y)$ is called a \textit{bound vector} and is denoted by $\overrightarrow{xy}$. Indeed, $X^2=\{\overrightarrow{xy}: x,y \in X\}$.  For each $x \in X$, we apply  $\mathbf{0}_x:=\overrightarrow{xx}$ as \textit{zero bound vector} at $x$ and $-\overrightarrow{xy}$ as the bound vector $\overrightarrow{yx}$. The bound vectors $\overrightarrow{xy}$ and $\overrightarrow{uz}$ are called \textit{admissible} if $y=u$. Therefore the sum of two admissible bound vectors $\overrightarrow{xy}$ and $\overrightarrow{yz}$ is defined by $\overrightarrow{xy}+\overrightarrow{yz}=\overrightarrow{xz}$.
 Ahmadi Kakavandi and Amini in \cite{KakavandiAmini} have introduced the \textit{dual space} of an Hadamard space, by using the concept of quasilinearization of abstract metric spaces presented by Berg and Nikolaev  in \cite{BergNikolaev}. The \textit{quasilinearization map} is defined as following:
\begin{align}\label{inb}
\langle \cdot,\cdot\rangle &: X^2 \times X^2 \rightarrow \mathbb{R}\\ \nonumber
\langle \overrightarrow{ab}, \overrightarrow{cd} \rangle & :=\frac{1}{2}\big\{d(a,d)^2+d(b,c)^2-d(a,c)^2-d(b,d)^2\big \}; ~a,b,c,d \in X.
\end{align}
Let $x,\,y \in X$, we define the mapping $\varphi_{\overrightarrow{xy}}: X\rightarrow \mathbb{R}$ by $\varphi_{\overrightarrow{xy}}(z)=\frac{1}{2} (d(x,z)^2-d(y,z)^2 )$; for each $z\in X$.
We will see that  $\varphi_{\overrightarrow{xy}}$ possess attractive properties that simplify some calculations. We observe that \eqref{inb} can be rewritten as following:
\begin{equation*}
\langle \overrightarrow{ab}, \overrightarrow{cd} \rangle=\varphi_{\overrightarrow{cd}}(b)-\varphi_{\overrightarrow{cd}}(a)=\varphi_{\overrightarrow{ab}}(d)-\varphi_{\overrightarrow{ab}}(c).
\end{equation*}
The metric space $(X,d)$ satisfies the \textit{Cauchy-Schwarz inequality} if
\[\langle \overrightarrow{ab}, \overrightarrow{cd} \rangle \leq d(a,b)d(c,d)~~\text{for all} ~a,b,c,d \in X.\]
This inequality characterizes $\cat(0)$ spaces. Indeed, it follows from \cite[Corollary 3]{BergNikolaev} that  a geodesic space $(X,d)$ is a $\cat(0)$ space if  and only if it satisfies in the Cauchy-Schwarz inequality.
For an Hadamard space $(X,d)$, consider the mapping
\begin{align*}
\Psi &:\mathbb{R} \times X^2 \rightarrow C(X,\mathbb{R})\\
& (t,a,b) \mapsto\Psi  (t,a,b)x=t \langle \overrightarrow{ab}, \overrightarrow{ax} \rangle; ~a,b,x \in X , t \in \mathbb{R},
\end{align*}
where $ C(X,\mathbb{R}) $ denotes the space of all continuous real-valued functions on $\mathbb{R} \times X^2 $. It follows from Cauchy-Schwarz inequality that $  \Psi  (t,a,b) $ is a Lipschitz function with Lipschitz semi-norm
\begin{equation}\label{lips22}
L(\Psi (t,a,b) )=|t|d(a,b) ,~ \text{for all}~ a,b \in X,~\text{and all}~ t \in \mathbb{R},
\end{equation}
where  the \textit{Lipschitz semi-norm} for any function $\varphi : (X, d) \rightarrow \mathbb{R} $ is defined by
\[ L(\varphi)=\sup\bigg\{ \frac{\varphi(x)-\varphi(y)}{d(x,y)}:x,y \in X, x\neq y\bigg\}.\]
 A \textit{pseudometric $D$ on $\mathbb{R} \times X^2$ induced by the Lipschitz semi-norm} \eqref{lips22}, is defined by
\[D((t,a,b),(s,c,d))=L(\Psi (t,a,b)-\Psi (s,c,d));~ a,b,c,d \in X,  t,s \in \mathbb{R}.\]
For an Hadamard space $(X,d)$, the pseudometric space $ (\mathbb{R} \times X^2,D)$ can be considered as a subspace of the pseudometric space of all real-valued Lipschitz functions Lip$(X,\mathbb{R})$. Note that, in view of \cite[Lemma 2.1]{KakavandiAmini}, $ D((t,a,b),(s,c,d))=0 ~\text{if and only if} ~t\langle \overrightarrow{ab},\overrightarrow{xy} \rangle =s\langle \overrightarrow{cd},\overrightarrow{xy} \rangle~\text{for all}~ x,y \in X $. Thus, $D$ induces an equivalence relation on $\mathbb{R} \times X^2$, where the equivalence class of $(t,a,b) \in\mathbb{R} \times X^2 $ is
\[[t\overrightarrow{ab}]=\{s\overrightarrow{cd}: s\in\mathbb{R}, c,d\in X, D((t,a,b),(s,c,d))=0\}. \]
The \textit{dual space} of an Hadamard space $(X,d)$, denoted by $X^*$, is the set of all equivalence classes  $[t\overrightarrow{ab}]$ where $(t,a,b) \in \mathbb{R} \times X^2$,  with the metric $D([t\overrightarrow{ab}],[s\overrightarrow{cd}]):=D((t,a,b),(s,c,d))$.  Clearly, the definition of equivalence classes implies that $[\overrightarrow{aa}]=[\overrightarrow{bb}]$ for all $a,b \in X$. The \textit{zero element} of $X^*$ is  $\mathbf{0}:=[t\overrightarrow{aa}]$, where $a \in X$ and $t \in \mathbb{R}$ are arbitrary. It is easy to see that the evaluation $\langle \mathbf{0}, \cdot \rangle$ vanishes for any bound vectors in $X^2$. Note that in general $X^*$ acts on $X^2$ by
\[ \langle x^*,\overrightarrow{xy} \rangle =t \langle \overrightarrow{ab}, \overrightarrow{xy} \rangle,~\text{where}~x^*=[t\overrightarrow{ab}] \in X^*~\text{and}~ \overrightarrow{xy} \in X^2.\]
The following notation will be used throughout this paper.
\[\Big \langle \sum_{i=1}^n\alpha_ix_i^*,\overrightarrow{xy} \Big \rangle:=\sum_{i=1}^n\alpha_i\langle x_i^*,\overrightarrow{xy}\rangle,~\alpha_i \in \mathbb{R},\, x_i^* \in X^*,\, n\in \mathbb{N},\, x, y\in X.\]
For an Hadamard space $(X,d)$, Chaipunya and Kumam in \citep{ChaipunyaKumam}, defined the \textit{linear dual space} of $X$  by
\[X^\loze=\bigg \{\sum_{i=1}^n\alpha_i x_i^* :\alpha_i \in \mathbb{R},\, x_i^* \in X^*,\, n \in \mathbb{N}\bigg \}.\]
Therefore, $X^\loze={\spa}\, X^*$. It is easy to see that $X^\loze$ is a normed space with the norm $\|x^\loz\|_{\loz}=L(x^\loz)$ for all $x^\loz \in X^\loze$. Indeed:
\begin{lemma} {\rm\cite[Proposition 3.5]{ZamaniRaeisi}}\label{dg} Let $X$ be an Hadamard space with linear dual space $X^\loze$. Then \begin{align*}
\|x^\loz \|_{\loz}:=\sup\Bigg\{\frac{\big|\langle x^\loz,\overrightarrow{ab}\rangle-\langle x^\loz,\overrightarrow{cd}\rangle\big|}{d(a,b)+d(c,d)}: a,b,c,d \in X, (a,c)\neq (b,d) \Bigg\},
\end{align*}
is a norm on $X^\loze$. In particular, $\|[t\overrightarrow{ab}]\|_{\loz}=|t|d(a,b)$.
\end{lemma}


\section{Flat Hadamard Spaces and $\mathcal{W}$-property}
Let $M$ be a relation from $X$ to $X^\loze$; i.e., $M\subseteq X\times X^\loze$. The \textit{domain} and \textit{range} of $M$ are defined, respectively, by
\begin{equation*}
\Dom(M):=\big\{ x\in X : \exists\, x^\loz\in X^\loze ~\text{such that}~ (x, x^\loz)\in M\big\},
\end{equation*}
and
\begin{equation*}
\Ran(M):=\big\{ x^\loz\in X^\loze : \exists\, x\in X ~\text{such that}~ (x, x^\loz)\in M\big\}.
\end{equation*}

\begin{definition}{\rm \label{zsd} Let $X$ be an Hadamard space with linear dual space $X^\loze$. We say that $ M \subseteq X \times X^\loze$ satisfies} \text{$\mathcal{W}$-property} {\rm if there exists $p\in X$ such that the following holds: }
\begin{equation*}
\big \langle x^\loz,\overrightarrow{p((1-\lambda)x_1\oplus \lambda x_2)} \big \rangle \leq(1-\lambda) \langle x^\loz,\overrightarrow{px_1}\rangle +\lambda \langle x^\loz,\overrightarrow{px_2} \rangle ,~\forall \lambda \in [0,1],\, \forall x^\loz \in\Ran(M),\, \forall x_1,x_2 \in\Dom(M).
\end{equation*}
\end{definition}
\begin{proposition}\label{inq2}
Let $X$ be an Hadamard space with linear dual space $X^\loze$ and let $ M \subseteq X \times X^\loze$. Then the following statements are equivalent:
\begin{description}
\item[(i)] $ M \subseteq X \times X^\loze$ satisfies the $\mathcal{W}$-property for some $p\in X$.
\item[(ii)] $ M \subseteq X \times X^\loze$ satisfies the $\mathcal{W}$-property for any $q\in X$.
\item[(iii)] For any $q\in X$,
\begin{align} \label{wgeneral}
\big\langle x^\loze,\overrightarrow{q(\oplus_{i=1}^n\lambda_ix_i)}\big \rangle \leq \sum_{i=1}^n\lambda_i\langle x^\loz,\overrightarrow{qx_i} \rangle,~ \text{for all}~ x^\loz \in\Ran(M), \{x_i\}_{i=1}^n \subseteq \Dom(M), \{\lambda_i\}_{i=1}^n \subseteq [0,1].\tag{$\mathcal{W}_n(q)$}
\end{align}
\item[(iv)]  For some $p\in X$,  $(\mathcal{W}_n(p))$ holds.
\end{description}
\end{proposition}
\begin{proof}~\begin{description}
\item[(i) $\Rightarrow$ (ii):] Let $q\in X$ be any arbitrary element of $X$, $\lambda \in [0,1]$,\, $x^\loz \in\Ran(M)$, and $x_1,x_2 \in\Dom(M)$. Then
\begin{align*}
\big \langle x^\loz,\overrightarrow{q((1-\lambda)x_1\oplus \lambda x_2)} \big \rangle &=\big \langle x^\loz,\overrightarrow{qp}+\overrightarrow{p((1-\lambda)x_1\oplus \lambda x_2)} \big \rangle \\
&= \langle x^\loz,\overrightarrow{qp} \rangle +\big \langle x^\loz,\overrightarrow{p((1-\lambda)x_1\oplus \lambda x_2)} \big \rangle\\
&\leq(1-\lambda  )(\langle x^\loz,\overrightarrow{qp}\rangle + \langle x^\loz,\overrightarrow{px_1}\rangle)+\lambda(\langle x^\loz,\overrightarrow{qp}\rangle+\langle x^\loz,\overrightarrow{px_2} \rangle)\\
&=(1-\lambda  )\langle x^\loz,\overrightarrow{qp} +\overrightarrow{px_1} \rangle +\lambda \langle x^\loz,\overrightarrow{qp}+\overrightarrow{px_2}  \rangle\\
&=(1-\lambda  )\langle x^\loz,\overrightarrow{qx_1} \rangle +\lambda\langle x^\loz,\overrightarrow{qx_2} \rangle,
\end{align*}
as required.
\item[(ii) $\Rightarrow$ (iii):] We proceed by induction on $n$. By Definition \ref{zsd}  the claim is true for $n=2$. Now assume that   ($\mathcal{W}_{n-1}(q)$) is true. In view of equation \eqref{cnv},
\begin{align*}
\big \langle x^\loze,\overrightarrow {q(\oplus_{i=1}^n\lambda_ix_i)}\big \rangle &=\Big \langle x^\loz,\overrightarrow{q((1-\lambda_n)\Big(\frac{\lambda_1}{1-\lambda_n}x_1\oplus \frac{\lambda_2}{1-\lambda_n}x_2\oplus \cdots \oplus\frac{\lambda_{n-1}}{1-\lambda_n}x_{n-1}\Big)\oplus\lambda_nx_n)}\Big \rangle\\
&\leq (1-\lambda_n)\Big \langle x^\loz,\overrightarrow{q \Big(\frac{\lambda_1}{1-\lambda_n}x_1\oplus \frac{\lambda_2}{1-\lambda_n}x_2 \oplus \cdots \oplus \frac{\lambda_{n-1}}{1-\lambda_n}x_{n-1}}\Big)\Big \rangle+\lambda_n\langle x^\loz,\overrightarrow{qx_n} \rangle \\
&\leq(1-\lambda_n)\sum_{i=1}^n\frac{\lambda_i}{1-\lambda_n}\langle x^\loz,\overrightarrow{qx_i}\rangle+\lambda_n\langle x^\loz,\overrightarrow{qx_n} \rangle\\
&=\sum_{i=1}^{n-1}\lambda_i\langle x^\loz,\overrightarrow{qx_i} \rangle+\lambda_n\langle x^\loz,\overrightarrow{qx_n} \rangle\\
&=\sum_{i=1}^n\lambda_i\langle x^\loz,\overrightarrow{qx_i} \rangle.
\end{align*}
\item[(iii) $\Rightarrow$ (iv):] Clear.
\item[(iv) $\Rightarrow$ (i):]  Take $n=2$ in $(\mathcal{W}_n(p))$.
\end{description}
We are done.
\end{proof}
\begin{remark}
It should be noticed that Proposition \ref{inq2} implies that $\mathcal{W}$-property is independent of the choice of the element $p \in X$.
\end{remark}
\begin{definition}{\rm{\rm\cite[Definition 3.1]{MovahediBehmardiSoleimani-Damaneh}}
An Hadamard space $(X,d)$ is said to be \textit{flat} if equality holds in the CN-inequality, i.e., for each $x,y\in X$ and $\lambda \in [0,1]$, the following holds:
\begin{align*}
d(z,(1-\lambda)x\oplus \lambda y)^2=(1-\lambda)d(z,x)^2 +\lambda d(z,y)^2-\lambda(1-\lambda)d(x,y)^2, ~ \textit{for all}  ~z\in X.
\end{align*}
}\end{definition}

\begin{proposition}  \label{frtg} Let $X$ be an Hadamard space. The following statements are equivalent:
\begin{description}
\item[{(i)}] $X$ is a flat Hadamard space.

\item[{(ii)}] $\langle \overrightarrow{x((1-\lambda)x\oplus \lambda y)},\overrightarrow{ab} \rangle=\lambda \langle \overrightarrow{xy},\overrightarrow{ab} \rangle$, for all $a,b,x,y\in X$ and all $\lambda\in[0,1]$.

\item[{(iii)}] $X\times X^\loze$ has $\mathcal{W}$-property.

\item[{(iv)}] Any subset of $X\times X^\loze$ has $\mathcal{W}$-property.

\item[{(v)}] For each $p,z \in X$, the mapping $\varphi_{\overrightarrow{pz}}$ is convex.

\item[{(vi)}] For each $p,z \in X$, the mapping $\varphi_{\overrightarrow{pz}}$ is affine, in the sense that:
\begin{align*}
\varphi_{\overrightarrow{pz}}((1-\lambda)x\oplus \lambda y)=(1-\lambda)\varphi_{\overrightarrow{pz}}(x)+\lambda\varphi_{\overrightarrow{pz}}(y),~\forall\,x, y\in X, \forall\,\lambda\in[0,1].
\end{align*}

\end{description}
\end{proposition}
\begin{proof}~\begin{description}
\item[(i) $\Leftrightarrow$ (ii):]  \cite[Theorem 3.2]{MovahediBehmardiSoleimani-Damaneh}.

\item[(ii)$\Rightarrow$  (iii):]
 Let $x,y \in X$, $\lambda \in [0,1]$ and $(x, x^\loz)\in X\times X^\loze$. Then $x^\loz =\sum^n_{i=1} \alpha_i[t_i\overrightarrow{a_ib_i}]\in X^\loze$,  and hence by using (ii) we get:
\begin{align*}
\big\langle x^\loz,\overrightarrow{p\big((1-\lambda)x\oplus \lambda y\big)}\big \rangle
&=\sum^n_{i=1} \alpha_i t_i \big \langle \overrightarrow{a_ib_i},\overrightarrow{px}+ \overrightarrow{x\big((1-\lambda)x\oplus \lambda y\big)}\big \rangle  \\
&=\sum^n_{i=1} \alpha_i t_i \big(\big \langle \overrightarrow{a_ib_i},\overrightarrow{px} \big \rangle +\big \langle \overrightarrow{a_ib_i},\overrightarrow{x\big((1-\lambda)x\oplus \lambda y\big)}\big \rangle \big)  \\
&=\sum^n_{i=1} \alpha_i t_i \big(\big \langle \overrightarrow{a_ib_i},\overrightarrow{px} \big \rangle +\lambda \big \langle \overrightarrow{a_ib_i},\overrightarrow{xy}\big \rangle \big)  \\
&=\sum^n_{i=1} \alpha_i t_i \big(\big \langle \overrightarrow{a_ib_i},\overrightarrow{px} \big \rangle +\lambda \big \langle \overrightarrow{a_ib_i},\overrightarrow{py}-\overrightarrow{px}\big \rangle \big)   \\
&=\sum^n_{i=1} \alpha_i t_i \big( (1-\lambda )\big\langle \overrightarrow{a_ib_i},\overrightarrow{px}\big  \rangle +\lambda\big \langle \overrightarrow{a_ib_i},\overrightarrow{py} \big\rangle \big)\\
&=(1-\lambda )\sum^n_{i=1} \alpha_i t_i\big \langle \overrightarrow{a_ib_i},\overrightarrow{px}\big \rangle +\lambda \sum^n_{i=1} \alpha_i t_i  \big  \langle \overrightarrow{a_ib_i},\overrightarrow{py} \big \rangle  \\
&=(1-\lambda) \Big \langle \sum^n_{i=1} \alpha_i[t_i\overrightarrow{a_ib_i}],\overrightarrow{px} \Big \rangle +\lambda \Big \langle \sum^n_{i=1} \alpha_i [t_i\overrightarrow{a_ib_i}],\overrightarrow{py}\Big \rangle \\
&=(1-\lambda) \langle x^\loz,\overrightarrow{px} \rangle +\lambda \langle x^\loz , \overrightarrow{py} \rangle.
\end{align*}

Therefore $ X \times X^\loze$ has $\mathcal{W}$-property.

\item[(iii) $\Leftrightarrow$  (iv):] Straightforward.

\item[(iv) $\Rightarrow$  (v):] Let $x,y \in X$ and $\lambda \in [0,1]$, then
\begin{align*}
\hspace{-1.5cm}(1-\lambda) \varphi_{\overrightarrow{pz}}(x)+\lambda \varphi_{\overrightarrow{pz}}(y)-\varphi_{\overrightarrow{pz}}((1-\lambda)x\oplus \lambda y)&= \lambda(\varphi_{\overrightarrow{pz}}(y)-\varphi_{\overrightarrow{pz}}(x))+\varphi_{\overrightarrow{pz}}(x)-\varphi_{\overrightarrow{pz}}((1-\lambda)x\oplus \lambda y)\\
&= \lambda \langle  \overrightarrow{pz},\overrightarrow{xy}\rangle +\langle \overrightarrow{pz},\overrightarrow{((1-\lambda)x\oplus \lambda y)x}\rangle\\
&= \lambda \langle  \overrightarrow{pz},\overrightarrow{py}-\overrightarrow{px}\rangle +\langle \overrightarrow{pz},\overrightarrow{px}-\overrightarrow{p((1-\lambda)x\oplus \lambda y)}\rangle\\
&=\lambda \langle  \overrightarrow{pz},\overrightarrow{py}\rangle +(1-\lambda)\langle  \overrightarrow{pz},\overrightarrow{px}\rangle-\langle  \overrightarrow{pz},\overrightarrow{p((1-\lambda)x\oplus \lambda y)}\rangle\\
&\geq 0.
\end{align*}
Therefore, $\varphi_{\overrightarrow{pz}}$ is convex.
\item[(v) $\Rightarrow$  (vi):] It is easy.

\item[(vi) $\Rightarrow$  (iii):]
Let $x,y, p \in X$, $\lambda \in [0,1]$ and  $x^\loz =\sum^n_{i=1} \alpha_i[t_i\overrightarrow{p_iz_i}]\in X^\loze$ be given. Then
\begin{align*}
\hspace{-2cm} \big\langle x^\loz,\overrightarrow{p\big((1-\lambda)x\oplus \lambda y\big)}\big \rangle
&=\big \langle\sum^n_{i=1} \alpha_ix^*_i,\overrightarrow{p\big((1-\lambda)x\oplus \lambda y\big)}\big \rangle  \\
&=\sum^n_{i=1} \alpha_i t_i \big \langle \overrightarrow{p_iz_i}, \overrightarrow{p\big((1-\lambda)x\oplus \lambda y\big)}\big \rangle  \\
&=\sum^n_{i=1}\alpha_i t_i \big(\varphi_{\overrightarrow{p_iz_i}}((1-\lambda)x\oplus \lambda y)- \varphi_{\overrightarrow{p_iz_i}}(p)\big) \\
&= \sum^n_{i=1} \alpha_i t_i\big((1-\lambda)\varphi_{\overrightarrow{p_iz_i}}(x)+\lambda \varphi_{\overrightarrow{p_iz_i}}(y)-\varphi_{\overrightarrow{p_iz_i}}(p)\big)\\
&=\sum^n_{i=1} \alpha_i t_i\Big((1-\lambda)(\varphi_{\overrightarrow{p_iz_i}}(x)-\varphi_{\overrightarrow{p_iz_i}}(p))+\lambda (\varphi_{\overrightarrow{p_iz_i}}(y)-\varphi_{\overrightarrow{p_iz_i}}(p))\Big)\\
&=\sum^n_{i=1} \alpha_i t_i \big( (1-\lambda )\big\langle \overrightarrow{p_iz_i},\overrightarrow{px}\big  \rangle +\lambda\big \langle \overrightarrow{p_iz_i},\overrightarrow{py} \big\rangle \big)\\
&=(1-\lambda )\sum^n_{i=1} \alpha_i t_i\big \langle \overrightarrow{p_iz_i},\overrightarrow{px}\big \rangle +\lambda \sum^n_{i=1} \alpha_i t_i  \big  \langle \overrightarrow{p_iz_i},\overrightarrow{py} \big \rangle  \\
&=(1-\lambda) \Big \langle \sum^n_{i=1} \alpha_i[t_i\overrightarrow{p_iz_i}],\overrightarrow{px} \Big \rangle +\lambda \Big \langle \sum^n_{i=1} \alpha_i [t_i\overrightarrow{p_iz_i}],\overrightarrow{py}\Big \rangle \\
&=(1-\lambda) \langle x^\loz,\overrightarrow{px} \rangle +\lambda \langle x^\loz , \overrightarrow{py} \rangle;
\end{align*}
i.e.,   $X\times X^\loze$ has $\mathcal{W}$-property.
\item[(iii)$\Rightarrow$ (ii):] For $a, b, x, y \in X$ and $\lambda \in [0,1]$, we have:
\begin{align*}
\lambda \langle \overrightarrow{ab},\overrightarrow{xy}\rangle-\langle \overrightarrow{ab},\overrightarrow{x((1-\lambda)x\oplus \lambda y)}\rangle&=
 \lambda \big(\langle \overrightarrow{ab},\overrightarrow{py}-\overrightarrow{px} \rangle \big)-\langle \overrightarrow{ab},\overrightarrow{p((1-\lambda)x\oplus \lambda y)}-\overrightarrow{px}  \rangle \\
 &= \lambda \big(\langle \overrightarrow{ab},\overrightarrow{py}\rangle - \langle \overrightarrow{ab},\overrightarrow{px} \rangle \big)-\langle \overrightarrow{ab},\overrightarrow{p((1-\lambda)x\oplus \lambda y)}  \rangle+ \langle \overrightarrow{ab},\overrightarrow{px}\rangle\\
 &=  (1-\lambda) \langle \overrightarrow{ab},\overrightarrow{px}\rangle +\lambda \langle \overrightarrow{ab},\overrightarrow{py} \rangle-\langle \overrightarrow{ab},\overrightarrow{p((1-\lambda)x\oplus \lambda y)}  \rangle\\
& =(1-\lambda) \langle x^\loz,\overrightarrow{px} \rangle +\lambda \langle x^\loz , \overrightarrow{py} \rangle \rangle-\langle x^\loz,\overrightarrow{p((1-\lambda)x\oplus \lambda y)},
\end{align*}
where $x^\loz=[\overrightarrow{ab}] \in X^\loze$. Since $ X \times X^\loze$ has $\mathcal{W}$-property, one can deduce that:
\begin{align}\label{z3}
\lambda \langle \overrightarrow{ab},\overrightarrow{xy}\rangle\geq\langle \overrightarrow{ab},\overrightarrow{x((1-\lambda)x\oplus \lambda y)}\rangle.
\end{align}
Hence, by interchanging the role of $a$ and $b$ in \eqref{z3}, we obtain:
\begin{align}
\langle \overrightarrow{ab},\overrightarrow{x((1-\lambda)x\oplus \lambda y)}\rangle \geq \lambda \langle \overrightarrow{ab},\overrightarrow{xy}\rangle.  \label{z33}
\end{align}
Finally, \eqref{z3} and \eqref{z33} yield:
\[\langle \overrightarrow{ab},\overrightarrow{p((1-\lambda)x\oplus \lambda xy)}  \rangle = \lambda \big(\langle \overrightarrow{ab},\overrightarrow{xy}\rangle \big).\]

\end{description}
We are done.
\end{proof}
~\\
The next example shows that there exists a relation $M \subseteq X \times X^\loze$ in the non-flat Hadamard spaces which doesn't have the $\mathcal{W}$-property.
\begin{example}\label{ExMonS} Consider the following equivalence relation on $\mathbb{N} \times [0,1]$:
\[(n,t)\sim (m,s) \Leftrightarrow t=s=0 ~\text{or} ~(n,t)=(m,s).\]
Set $ X:=\frac{\mathbb{N} \times [0,1]}{\sim} $ and let $ d :X \times X \rightarrow \mathbb{R} $ be defined by
\begin{align*}
d([(n,t)],[(m,s)])=\begin{cases}
|t-s| &n=m,
\\
t+s  & n\neq m.
\end{cases}
\end{align*}
The geodesic joining $x=[(n,t)]$ to $y=[(m,s)] $ is defined as follows:
\[ (1-\lambda)x \oplus \lambda y:=\begin{cases}
[(n,(1-\lambda)t-\lambda s)] & 0 \leq \lambda \leq \frac{t}{t+s},
\\[1mm]
[(m,(\lambda-1)t+\lambda s)]   & \frac{t}{t+s} \leq  \lambda \leq 1,
\end{cases}
\]
whenever $x\neq y$ and vacuously  $(1-\lambda)x \oplus \lambda x:=x$.
It is known that (see \cite[Example 4.7]{Kakavandi}) $(X,d)$ is an $\mathbb{R}$-tree space. It follows from \cite[Example 1.2.10]{Bacak2014}, that any $\mathbb{R}$-tree space is an Hadamard space.
 Let $x=[(2,\frac{1}{2})]$, $ y=[(1,\frac{1}{2})] $, $ a=[(3,\frac{1}{3})] $, $b= [(2,\frac{1}{2})] $ and $\lambda = \frac{1}{5}$. Then $\frac{4}{5}x \oplus \frac{1}{5}y =[(2,\frac{3}{10})]$ and
\[\bigg\langle \overrightarrow{x(\frac{4}{5}x \oplus \frac{1}{5}y )},\overrightarrow{ab}\bigg \rangle=\frac{-1}{6} \neq\frac{-1}{10}= \frac{1}{5}\Big\langle\overrightarrow{xy}, \overrightarrow{ab}\Big\rangle.\]
Now, Proposition \ref{frtg}(ii) implies that  $(X,d)$ is not a flat Hadamard space.
For each $n\in \mathbb{N}$, set $x_n:=[(n,\frac{1}{2})]$ and  $y_n:=[(n,\frac{1}{n})]$. Now, we define  \[M:=\big\{(x_n,[\overrightarrow{y_{n+1}y_n}]):n\in \mathbb{N}\big\}\subseteq X\times X^\loze.\]
 Take $p=[(1,1)] \in X$, $[\overrightarrow{y_5y_4}] \in \Ran(M)$ and $\lambda=\frac{1}{3}$. Clearly, $\tilde{x}:=(1-\lambda)x_1 \oplus \lambda x_3=[(1,\frac{1}{6})]$ and
$\big \langle[\overrightarrow{y_5y_4}],\overrightarrow{p\tilde{x}} \big \rangle=\frac{1}{24},$
while,
\[\frac{2}{3}\langle [\overrightarrow{y_5y_4}],\overrightarrow{px_1}\rangle +\frac{1}{3} \langle [\overrightarrow{y_5y_4}],\overrightarrow{px_3} \rangle=\frac{1}{40}.\]
Therefore, $M $ doesn't have  the $\mathcal{W}$-property.
\end{example}
\section{Monotone Relations}

Ahmadi Kakavandi and  Amini \cite{KakavandiAmini} introduced the notion of monotone operators in Hadamard spaces. In \cite{KhatibzadehRanjbar2017},
Khatibzadeh and Ranjbar, investigated some properties of monotone operators and their resolvents and also proximal point algorithm in Hadamard spaces. Chaipunya and  Kumam \cite{ChaipunyaKumam} studied the general proximal point method for finding a zero point of a maximal monotone set-valued vector field defined on Hadamard spaces. They proved the relation between the maximality and Minty's surjectivity condition. Zamani Eskandani and Raeisi \cite{ZamaniRaeisi}, by using products of finitely many resolvents of monotone operators, proposed an iterative algorithm for finding a common zero of a finite family of monotone operators and a common fixed point of an infinitely countable family of non-expansive mappings in Hadamard spaces.
In this section, we will characterize the notation of monotone relations in Hadamard spaces based on characterization of monotone sets in Banach spaces \cite{CoodeySimons1996, Simons1998, Zalinescu}.

\begin{definition}{\rm Let $X$ be an Hadamard space with linear dual space $X^\loze$. The set $M \subseteq X \times X^\loze$ is called \textit{monotone} if  $\langle x^\loz-y^\loz,\overrightarrow{yx}\rangle\geq0$,~for all $(x,x^\loz)$,$(y,y^\loz)$ in $M$.
}\end{definition}
\begin{example} Let $x_n$, $y_n$ and  $M$ be  the same as in Example \ref{ExMonS}. Let $(u, u^\loz), (v, v^\loz)\in M$. There exists $m,n\in \mathbb{N}$ such that $u=x_n$, $u^\loz:=[\overrightarrow{y_{n+1}y_n}]$, $v=x_m$ and $v^\loz:=[\overrightarrow{y_{m+1}y_m}]$. Then
\begin{align*}
\langle u^\loz-v^\loz, \overrightarrow{vu} \rangle = \langle u^\loz, \overrightarrow{vu} \rangle -\langle v^\loz, \overrightarrow{vu} \rangle&= \Big\langle \big[\overrightarrow{\big[(n+1,\frac{1}{n+1})\big]\big[(n,\frac{1}{n})\big]}\big], \overrightarrow{\big[(m,\frac{1}{2})\big]\big[(n,\frac{1}{2})\big]} \Big\rangle\\
&\quad -\Big\langle \big[\overrightarrow{\big[(m+1,\frac{1}{m+1})\big]\big[(m,\frac{1}{m})\big]}\big], \overrightarrow{\big[(m,\frac{1}{2})\big]\big[(n,\frac{1}{2})\big]}\Big \rangle \\
&=\begin{cases}
0,  &~~~~n=m,\\[1mm]
\frac{1}{m+1}+\frac{1}{n}+\frac{1}{m}, &~~~~n=m+1, \\[1mm]
\frac{1}{n+1}+\frac{1}{n}+\frac{1}{m},  &~~~~n=m-1, \\[1mm]
\frac{1}{n}+\frac{1}{m},  &~~~~n\notin \{m-1,m,m+1\}.
\end{cases}
\end{align*}
Therefore, $\langle u^\loz-v^\loz, \overrightarrow{vu} \rangle \geq 0$ which shows that, $M$ is a monotone relation.
\end{example}

In the sequel, we need the following notations. Let $X$ be an Hadamard space and  $Y \subseteq X$. Put
\[\varsigma_Y:=\bigg\{\eta:Y\rightarrow[0,+\infty[ ~\big|~ \supp\,\eta \textit{~is finite} \text{~and~} \sum_{x\in Y}\eta (x)=1\bigg\}\] where $\supp\,{\eta}=\{y\in Y:\eta (y)\neq0\}$. Clearly, for each $\emptyset \neq A \subset Y$, $\varsigma_A=\{\eta \in\varsigma_Y:\supp\,\eta  \subseteq A\}$. It is obvious that $\varsigma_A$ is a convex subset of $\mathbb{R}^Y$. Moreover, if $\emptyset \neq A\subseteq B$, then $\varsigma_A \subseteq \varsigma_B$.   Suppose $u\in Y$ be fixed. Define  $\delta_u  \in \varsigma_Y$  by \begin{equation*}
\delta_u (x)= \left\{
\begin{array}{rl}
    1  &   x=u,\\
    0  &   x\neq u.
\end{array} \right.
\end{equation*}
Let $M \subseteq X \times X^\loze$ and $\eta \in \varsigma_A$. Then $\supp\eta=\{\lambda_1,\ldots,\lambda_n\}$ where $\lambda_i=\eta (x_i,x_i^\loz),~\text{for each}~1\leq i \leq n $. Let  $p\in X$ be fixed. Define $\alpha :\varsigma_{X \times X^{\loze}} \rightarrow X$ $($resp. $\beta :\varsigma_{X \times X^\loze} \rightarrow  X^\loze$ and $\theta_p :\varsigma_{X \times X^{\loze}} \rightarrow \mathbb{R})$ by
\begin{align*}
\alpha(\eta)=\bigoplus_{i=1}^n \lambda_ix_i ,~\Big(\text{resp.}~\beta(\eta)=\sum_{i=1}^n \lambda_ix_i^\loz ~\text{and} ~\theta_p(\eta)=\sum_{i=1}^n \lambda_i\langle x_i^\loz,\overrightarrow{px_i} \rangle\Big).
\end{align*}

 \begin{proposition}\label{tetapw} Let $X$ be an Hadamard space, $M \subseteq X \times X^\loze$ and $p\in X$. Set
 \begin{equation}
 \Theta_{p,M}:=\Big\{\eta \in\varsigma_M:\theta_p (\eta ) \geq \langle \beta(\eta ),\overrightarrow{p\alpha(\eta )}\rangle\Big\}. \label{tet}
 \end{equation}
Then $ \Theta_{p,M}= \Theta_{q,M}$ for any $q\in X$.
 \end{proposition}
 \begin{proof}
It is enough to show that $ \Theta_{p,M} \subseteq \Theta_{q,M}$.  Let $\eta \in \Theta_{p,M}$ be such that $\supp\eta=\{\lambda_1,\ldots,\lambda_n\}$ where $\lambda_i=\eta (x_i,x_i^\loz),~\text{for each}~1\leq i \leq n $. Then
\begin{align*}
\theta_q (\eta ) &=\sum_{i=1}^n \lambda_i\langle x_i^\loz,\overrightarrow{qx_i} \rangle=\sum_{i=1}^n \lambda_i\langle x_i^\loz,\overrightarrow{qp} \rangle+\sum_{i=1}^n \lambda_i\langle x_i^\loz,\overrightarrow{px_i} \rangle\\
&=\langle \sum_{i=1}^n \lambda_ix_i^\loz,\overrightarrow{qp} \rangle+\theta_p (\eta )=\langle \beta(\eta ),\overrightarrow{qp}\rangle+\theta_p (\eta )\\
&\geq\langle \beta(\eta ),\overrightarrow{qp}\rangle+\langle \beta(\eta ),\overrightarrow{p\alpha(\eta )}\rangle\\
&=\langle \beta(\eta ),\overrightarrow{q\alpha(\eta )}\rangle.
\end{align*}
Therefore, $ \eta \in \Theta_{q,M}$, i.e., $ \Theta_{p,M}\subseteq \Theta_{q,M}$.
 \end{proof}
According to Proposition \ref{tetapw}, for each  $M \subseteq X \times X^\loze$,  the set $\Theta_{p,M}$  is independent of the choice of the element $p \in X$ and hence we denote the set $\Theta_{p,M}$ by $\Theta_M$.
\begin{theorem}\label{taw} Let $X$ be an Hadamard space and $M \subseteq X \times X^\loze$ satisfies the
 $\mathcal{W}$-property. Then $M$ is a monotone set  if and only if   $\Theta_M=\varsigma_M$.
\end{theorem}
\begin{proof} Let $M$ be a monotone set. In view of (\ref{tet}), it is enough to show that $ \varsigma_M \subseteq  \Theta_M$. Let $\eta \in \varsigma_M$ be such that $\supp\eta=\{\lambda_1,\ldots,\lambda_n\}$ where $\lambda_i=\eta (x_i,x_i^\loz),~\text{for each}~1\leq i \leq n $. By using Proposition \ref{inq2}, we obtain:
 \begin{align*}
\theta_p (\eta ) - \langle \beta(\eta ),\overrightarrow{p\alpha(\eta )}\rangle &=\sum_{i=1}^n \lambda_i\langle x_i^\loz,\overrightarrow{px_i} \rangle-\Big\langle \sum_{j=1}^n \lambda_jx_j^\loz ,\overrightarrow{p\big(\scalebox{1.4}{$\oplus$}_{i=1}^n  \lambda_ix_i\big)}  \Big\rangle\\
&=\sum_{i=1}^n \lambda_i\langle x_i^\loz,\overrightarrow{px_i} \rangle-\sum_{j=1}^n \lambda_j\Big \langle  x_j^\loz ,\overrightarrow{p\big(\scalebox{1.4}{$\oplus$}_{i=1}^n  \lambda_ix_i\big)}  \Big\rangle \\
& \geq \sum_{i=1}^n \lambda_i\langle x_i^\loz,\overrightarrow{px_i} \rangle-\sum_{j=1}^n\sum_{i=1}^n \lambda_i \lambda_j\langle x_j^\loz,\overrightarrow{px_i} \rangle \\
& = \sum_{j=1}^n\sum_{i=1}^n \lambda_i \lambda_j \langle x_i^\loz,\overrightarrow{px_i} \rangle-\sum_{j=1}^n\sum_{i=1}^n \lambda_i \lambda_j\langle x_j^\loz,\overrightarrow{px_i} \rangle\\
& = \sum_{j=1}^n\sum_{i=1}^n \lambda_i \lambda_j \langle x_i^\loz-x_j^\loz,\overrightarrow{px_i} \rangle\\
& = \sum_{j=1}^n\sum_{i=1}^n \lambda_i \lambda_j \langle x_j^\loz-x_i^\loz,\overrightarrow{px_j} \rangle\\
& =\frac{1}{2} \sum_{i=1}^n\sum_{j=1}^n \lambda_i \lambda_j \langle x_i^\loz-x_j^\loz,\overrightarrow{px_i}-\overrightarrow{px_j} \rangle\\
& =\frac{1}{2} \sum_{i=1}^n\sum_{j=1}^n \lambda_i \lambda_j \langle x_i^\loz-x_j^\loz,\overrightarrow{x_jx_i} \rangle\geq0.
\end{align*}
Then $\varsigma_M \subseteq  \Theta_M$ and hence  $ \varsigma_M =\Theta_M$. For the converse, let $(x,x^{\loz}),(y,y^{\loz}) \in M $ and set $\eta:=\frac{1}{2}\delta_{(x,x^{\loz})}+\frac{1}{2}\delta_{(y,y^{\loz})}\in \varsigma_M$. By using $\mathcal{W}$-property, we get:
\begin{align*}
\frac{1}{4}\langle x^\loz-y^\loz,\overrightarrow{yx} \rangle &=\frac{1}{4}\langle x^\loz-y^\loz,\overrightarrow{px}-\overrightarrow{py} \rangle\\
&=\frac{1}{4}(\langle x^\loz-y^\loz,\overrightarrow{px} \rangle-\langle x^\loz-y^\loz,\overrightarrow{py} \rangle)\\
&=\frac{1}{4}\langle x^\loz,\overrightarrow{px} \rangle+\frac{1}{4}\langle y^\loz,\overrightarrow{py} \rangle-\frac{1}{4}\langle y^\loz,\overrightarrow{px} \rangle-\frac{1}{4}\langle x^\loz,\overrightarrow{py} \rangle\\
&=\frac{1}{2}\langle x^\loz,\overrightarrow{px} \rangle+\frac{1}{2}\langle y^\loz,\overrightarrow{py} \rangle-\frac{1}{4}\langle x^\loz,\overrightarrow{px} \rangle-\frac{1}{4}\langle x^\loz,\overrightarrow{py} \rangle-\frac{1}{4}\langle y^\loz,\overrightarrow{px} \rangle-\frac{1}{4}\langle y^\loz,\overrightarrow{py} \rangle\\
&\geq \frac{1}{2}\langle x^\loz,\overrightarrow{px} \rangle+\frac{1}{2}\big\langle y^\loz,\overrightarrow{py} \big\rangle-\big\langle\frac{1}{2}x^\loz+ \frac{1}{2}y^\loz,\overrightarrow{p(\frac{1}{2}x\oplus\frac{1}{2}y)} \big\rangle \\
&=\frac{1}{2}\langle x^\loz,\overrightarrow{px} \rangle+\frac{1}{2}\langle y^\loz,\overrightarrow{py} \rangle-\frac{1}{2}\big\langle x^\loz,\overrightarrow{p(\frac{1}{2}x\oplus\frac{1}{2}y)} \big\rangle-\frac{1}{2}\big\langle y^\loz,\overrightarrow{p(\frac{1}{2}x\oplus\frac{1}{2}y)} \big\rangle\\
&=\theta_p (\eta ) - \langle \beta(\eta ),\overrightarrow{p\alpha(\eta )}\rangle \geq0.
\end{align*}
Therefore, $M $ is monotone.
\end{proof}
\begin{corollary}Let $X$ be a flat Hadamard space and $M \subseteq X \times X^\loze$. Then $M$ is a monotone set  if and only if $\Theta_M=\varsigma_M$.
\end{corollary}
\begin{proof} Since $X$ is flat, Proposition \ref{frtg} implies that $M \subseteq X \times X^\loze$ satisfies the $\mathcal{W}$-property.  Then the conclusion follows immediately from Theorem \ref{taw}.
\end{proof}

A fundamental result concerning monotone operators is the extension
theorem of Debrunner-Flor (for a proof see \cite[Theorem 4.3.1]{BurachikIusem2008} or 
\cite[Proposition 2.17]{Zeidler1986}). In the sequel, we prove a type of this result for monotone relations from an Hadamard space to its linear dual space. First, we recall some notions and results.

\begin{definition} {\rm{\rm\cite[Definition 2.4]{KakavandiAmini}}}{\rm\,Let $\{x_n\}$ be a sequence in an Hadamard space $X$. The sequence $\{x_n\}$ is said to be \textit{weakly convergent} to $x\in X$, denoted by $x_n\overset{w}\longrightarrow x$, if $\lim_{n \rightarrow \infty} \langle \overrightarrow{xx_n},\overrightarrow{xy} \rangle =0$, for all $y \in X$.
}\end{definition}
One can easily see that convergence in the metric implies weak convergence.
\begin{lemma} {\rm{\rm\cite[Proposition 3.6]{ZamaniRaeisi}}} \label{loz1}Let $\{x_n\} $ be a bounded sequence in  an Hadamard space $(X,d)$ with linear dual space $X^\loze$ and let $\{x^\loz_n\}$ be a sequence in $X^{\loze}$. If  $\{x_n\}$ is  weakly convergent to $x$ and $x^{\loz}_n\xrightarrow{\|\cdot\|_\loz}x^\loz$, then $\langle x^\loz_n,\overrightarrow{x_nz}\rangle \rightarrow \langle x^\loz,\overrightarrow{xz}\rangle $,  for all $z\in X$.
\end{lemma}

\begin{theorem} { \label{dflr}
Let $X$ be an Hadamard space and  $M \subseteq X \times X^\loze$ be a monotone relation satisfies the $\mathcal{W}$-property. Let $C \subseteq X^\loze$ be a compact and  convex set, and $\varphi: C \rightarrow X$ be a continuous function. Then there exists $z^\loz \in C$ such that $\{(\varphi(z^\loz ),z^\loz )\} \cup M$ is monotone.}
\end{theorem}
\begin{proof}Let $x\in X$, $u^\loz, v^\loz\in X^\loze$ be  arbitrary and fixed element. Consider the function $\tau:C\rightarrow \mathbb{R}$ defined by
\begin{align*}
   \tau(x^\loz)= \langle x^\loz-v^\loz,\overrightarrow{x\varphi(u^\loz)} \rangle,~x^\loz\in C.
\end{align*}
Let $\{x_n^\loz\} \subseteq C$ be such that  $x^\loz_n\xrightarrow{\|\cdot\|_\loz} x^\loz $, for some $x^\loz \in C$. By  Lemma \ref{loz1}, \[\langle x_n^\loz-v^\loz,\overrightarrow{x\varphi(u^\loz)}\rangle\rightarrow\langle x^\loz-v^\loz,\overrightarrow{x\varphi(u^\loz)}\rangle.\]
Thus $\tau(x_n^\loz)\rightarrow \tau(x^\loz)$. Hence $\tau$ is continuous.
For every $(y,y^\loz) \in M$, set \[U(y,y^\loz):=\{u^\loz \in C: \langle u^\loz-y^\loz,\overrightarrow{y\varphi(u^\loz)}\rangle <0\}.\]
Continuity of $\tau$ implies that $ U(y,y^\loz) $ is an open subset of $C$. Suppose that the conclusion fails. Then for each $u^\loz \in C$ there exists $(y,y^\loz) \in M$ such that $u^\loz \in U(y,y^\loz) $. This means that the family of open sets $\{U(y,y^\loz)\}_{(y,y^\loz) \in M}$ is an open cover of $C$. Using the compactness of $C$, we obtain that $C = \bigcup_{i=1}^n U(y_i,y_i^\loz)$. In addition, \cite[Page 756]{Zeidler1986} implies that there exists a partition of unity associated with this finite subcover. Hence, there are continuous functions  $\psi_i:X^\loze\rightarrow \mathbb{R} ~(1\leq i\leq n) $ satisfying
\begin{enumerate}
\item[(i)] $\sum _{i=1}^n \psi_i(x^\loz)=1$, for all $x^\loz\in C$.
\item[(ii)] $\psi_i(x^\loz)\geq 0$, for all  $x^\loz\in C$ and all  $i\in \{1,\ldots,n\}$.
\item[(iii)] $\{x^\loz \in C: \psi_i(x^\loz)> 0\} \subseteq U_i:=U(y_i,y_i^\loz) $ for all $i\in \{1,\ldots,n\}$.
\end{enumerate}
Set $K:=\co(\{y_1^\loz,\ldots,y_n^\loz\}) \subseteq C$ and define \begin{align*}
\iota :&K \rightarrow K\\
& u^\loz \mapsto \sum_{i=1}^n \psi_i(u^\loz)y_i^\loz.
\end{align*}
Let $\{u_m^\loz\} \subseteq K$ be such that $u_m^\loz \rightarrow u^\loz$, 
\begin{align*}
\Big\|\iota(u^\loz_m)-\iota(u^\loz) \Big\|_{\loz}&=\Big\|\sum_{i=1}^n \psi_i(u_m^\loz)y_i^\loz-\sum_{i=1}^n \psi_i(u^\loz)y_i^\loz \Big\|_{\loz}\\&=\Big \|\sum_{i=1}^n (\psi_i(u_m^\loz)- \psi_i(u^\loz))y_i^\loz \Big \|_{\loz}\\
 & \leq \sum_{i=1}^n \big \|(\psi_i(u_m^\loz)- \psi_i(u^\loz))y_i^\loz \big \|_{\loz} \\
 & \leq \sum_{i=1}^n \big|\psi_i(u_m^\loz)- \psi_i(u^\loz)\big| \| y_i^\loz \|_{\loz}.
\end{align*}
By continuity of $\psi_i~(1 \leq i \leq n)$, letting  $ m\rightarrow +\infty$, then $ \psi_i(u_m^\loz)\rightarrow \psi_i(u^\loz) $ and this implies that $ \iota (u_m^\loz)\rightarrow\iota (u^\loz) $ and so $ \iota $ is continuous.
One can identify $K$  with a finite-dimensional convex and compact set. By using Brouwer fixed point theorem \cite[Proposition 2.6]{Zeidler1986}, there exists $w^\loz \in K$ such that $\iota(w^\loz)=w^\loz$. Moreover, by using Proposition  \ref{inq2} we get:
\begin{align} \nonumber
0&=\Big  \langle \iota(w^\loz)-w^\loz,\overrightarrow{\varphi(w^\loz)(\oplus_{j} \psi_j(w^\loz)y_j)} \Big  \rangle\\ \nonumber
&=\Big  \langle \sum_i \psi_i(w^\loz)(y_i^\loz-w^\loz),\overrightarrow{\varphi(w^\loz)(\oplus_j \psi_j(w^\loz)y_j)} \Big  \rangle\\ \nonumber
&=\Big  \langle \sum_{i} \psi_i(w^\loz)(y_i^\loz-w^\loz),\overrightarrow{p(\oplus_{j} \psi_j(w^\loz)y_j}) \Big  \rangle-\Big  \langle \sum_{i}\psi_i(w^\loz)(y_i^\loz-w^\loz),\overrightarrow{p\varphi(w^\loz)} \Big  \rangle&(p\in X)\\ \nonumber
&\leq \sum_{j}\psi_j(w^\loz) \Big  \langle \sum_{i} \psi_i(w^\loz)(y_i^\loz-w^\loz),\overrightarrow{p y_j} \Big  \rangle-\Big  \langle \sum_{i}\psi_i(w^\loz)(y_i^\loz-w^\loz),\overrightarrow{p\varphi(w^\loz)} \Big  \rangle \\ \nonumber
&= \sum_{j}\psi_j(w^\loz) \Big  \langle \sum_{i} \psi_i(w^\loz)(y_i^\loz-w^\loz),\overrightarrow{p y_j} \Big  \rangle-\sum_{j}\psi_j(w^\loz) \Big  \langle \sum_{i}\psi_i(w^\loz)(y_i^\loz-w^\loz),\overrightarrow{p\varphi(w^\loz)} \Big  \rangle\\  \nonumber
&= \sum_{j}\psi_j(w^\loz) \Big  \langle \sum_{i} \psi_i(w^\loz)(y_i^\loz-w^\loz),\overrightarrow{p y_j}-\overrightarrow{p\varphi(w^\loz)} \Big  \rangle\\  \nonumber
&= \sum_{j}\psi_j(w^\loz) \Big  \langle \sum_{i} \psi_i(w^\loz)(y_i^\loz-w^\loz),\overrightarrow{\varphi(w^\loz)y_j} \Big  \rangle\\ \nonumber
&= \sum_{j}\psi_j(w^\loz)\sum_{i}\psi_i(w^\loz) \Big  \langle y_i^\loz-w^\loz,\overrightarrow{\varphi(w^\loz)y_j} \Big  \rangle\\ \nonumber
&= \sum_{j}\sum_{i}\psi_j(w^\loz)\psi_i(w^\loz) \Big  \langle y_i^\loz-w^\loz,\overrightarrow{\varphi(w^\loz)y_j} \Big  \rangle\\
&= \sum_{i,j}\psi_i(w^\loz)\psi_j(w^\loz) \Big  \langle y_i^\loz-w^\loz,\overrightarrow{\varphi(w^\loz)y_j} \Big  \rangle. \label{ineqm}
\end{align}
Set $a_{ij}=\langle y_i^\loz-w^\loz,\overrightarrow{\varphi(w^\loz)y_j} \rangle$. It follows from monotonicity of $M$ that
\begin{align*}
 a_{ii}+a_{jj}-a_{ij}-a_{ji}&=\langle y_i^\loz-w^\loz,\overrightarrow{\varphi(w^\loz)y_i} \rangle+\langle y_j^\loz-w^\loz,\overrightarrow{\varphi(w^\loz)y_j} \rangle\\ \nonumber
&\quad-\langle y_i^\loz-w^\loz,\overrightarrow{\varphi(w^\loz)y_j} \rangle-\langle y_j^\loz-w^\loz,\overrightarrow{\varphi(w^\loz)y_i} \rangle\\
&=\langle y_i^\loz-y_j^\loz\,,\overrightarrow{\varphi(w^\loz)y_i}-\overrightarrow{\varphi(w^\loz)y_j} \rangle\\
&=\langle y_i^\loz-y_j^\loz\,,\overrightarrow{y_jy_i}\rangle\geq0;
\end{align*}
i.e.,
\begin{align}\label{ineqm00}
 a_{ii}+a_{jj}\geq a_{ij}
+a_{ji}.
\end{align}
Applying \eqref{ineqm} and  \eqref{ineqm00}, we obtain:
\begin{align} \nonumber
0& \leq \sum_{i,j}^n\psi_i(w^\loz)\psi_j(w^\loz) a_{ij}\\ \nonumber
&=\sum^n_{i<j}\psi_i(w^\loz)\psi_j(w^\loz) a_{ij}+\sum_{i=1} ^n \psi_i(w^\loz)^2a_{ii}+\sum^n_{i>j}\psi_i(w^\loz)\psi_j(w^\loz) a_{ij}\\
&=\sum_{i=1}^n \psi_{i}(w^\loz)^2a_{ii}+\sum^n_{i<j}\psi_i(w^\loz)\psi_j(w^\loz) (a_{ij}+a_{ji})\\
&\leq \sum^n_{i=1} \psi_i(w^\loz)^2a_{ii}+\sum^n_{i<j}\psi_i(w^\loz)\psi_j(w^\loz) (a_{ii}+a_{jj}). \label{ineqs}
\end{align}
Set $I(w^\loz):=\big\{i\in \{1,\ldots,n\}:w^\loz \in U_i \big\}$. Applying property (iii) of the partition of unity  in \eqref{ineqs} we get:
\begin{align}
0\leq \sum_{i \in I(w^\loz)} \psi_i(w^\loz)^2a_{ii}+\mathop{\sum_{i<j}}_{i,j \in I(w^\loz)}\psi_i(w^\loz)\psi_j(w^\loz) (a_{ii}+a_{jj}). \label{ijj}
\end{align}
By using property (iii) of the partition of unity and the definition of $U_i$, one deduce that all terms in the right-hand side of \eqref{ijj} are nonpositive. So all of $\psi_i(w^\loz)$'s must be vanish, which contradicts with (i).
\end{proof}

\begin{corollary} Let $X$ be a flat Hadamard space and  $M \subseteq X \times X^\loze$ be a monotone set. Let $C \subseteq X^\loze$ be a compact and convex set, and $\varphi: C \rightarrow X$ be a continuous function. Then there exists $z^\loz \in C$ such that $\{(\varphi(z^\loz ),z^\loz )\} \cup M$ is monotone.
\end{corollary}
\begin{proof} Since $X$ is flat, it follows from Proposition \ref{frtg} that $M \subseteq X\times X^\loze$ has $\mathcal{W}$-property. The inclusion follows from Theorem \ref{dflr}.
\end{proof}

\end{document}